\documentclass[reqno,centertags,11pt,draft]{amsart}
\usepackage{amsmath,amsthm,amsfonts,amssymb,enumerate}
\usepackage[bookmarksopen=true,final]{hyperref}
\usepackage{mathrsfs}
\usepackage[nobysame]{amsrefs}
\usepackage{hyperref}
\usepackage[pdftex]{graphicx,color}
\usepackage{graphicx}
\usepackage{marvosym}
\usepackage{scalerel}
\usepackage{bbm}
\usepackage{bm}
\usepackage{setspace}
\usepackage{tikz}
\usepackage[foot]{amsaddr}
\usepackage[normalem]{ulem}

\usepackage{arydshln}

\usepackage{fancyhdr} 
\usepackage{ulem} 

\usepackage{xcolor}

\newtheorem{thm}{Theorem}[section]
\newtheorem{lem}[thm]{Lemma}

\newtheorem{cor}[thm]{Corollary}
\newtheorem{prop}[thm]{Proposition}
\newtheorem{rem}[thm]{Remark}
\newtheorem{defn}[thm]{Definition}

\newcommand{\R}{\mathbb{R}}
\newcommand{\C}{\mathbb C}

\newcommand{\Z}{\mathbb{Z}}
\newcommand{\N}{\mathbb{N}}
\newcommand{\T}{\mathbb{\partial\mathbb{D}}}
\newcommand{\D}{\mathbb{D}}

\renewcommand{\d}{\partial}
\newcommand{\supp}{\operatorname{supp}}

\DeclareMathOperator{\sgn}{sgn}

\newcommand{\abs}[1]{\left|#1\right|}
\newcommand{\set}[1]{\left\{#1\right\}}

\begin{document}
\title{Angelesco and AT systems on the unit circle 
}
\author{Rostyslav Kozhan$^1$}
\email{kozhan@math.uu.se}
\author{Marcus Vaktnäs$^{1}$}
\email{marcus.vaktnas@math.uu.se}
\address{$^{1}$Department of Mathematics, Uppsala University, S-751 06 Uppsala, Sweden}
\date{\today}
\begin{abstract}
    We introduce the concept of Laurent multiple orthogonality on the unit circle and define Angelesco and AT systems in this setting.
    Using a generalized Andreief identity, we establish normality of all multi-indices for any such system, thereby ensuring existence and uniqueness of Laurent multiple orthogonal polynomials of type I and type II at every location.
    As an application, we demonstrate existence and uniqueness of the approximants for two natural two-point Hermite-Padé problems -- type~I and type~II -- arising in the simultaneous rational approximation of $r$ Carath\'{e}odory functions.

\end{abstract}
\maketitle

\section{Introduction 
}


Given a system of probability measures $\bm{\mu}=(\mu_1,\ldots,\mu_r)$ on the real line and a multi-index $\bm{n} = (n_1,\dots,n_r) \in \N^r$ (where $\N=\set{0,1,2,\dots}$),  multiple orthogonal polynomials on the real line are defined as the 
univariate polynomials $P_{\bm{n}}(x)$ of degree at most $\abs{\bm{n}} := n_1 + \ldots + n_r$ that satisfy simultaneous orthogonality conditions
\begin{equation}\label{eq:MOPRL}
    \int_\R P_{\bm{n}}(x) x^k d\mu_j(x) = 0, \qquad k=0,1,\ldots,n_j-1,
\end{equation}
for each $j=1,\ldots,r$. If $r=1$ then the polynomials $P_{\bm{n}}(x)$ reduce to the usual orthogonal polynomials associated with a single measure on the real line.

Such polynomials arise naturally as solutions to the Hermite--Padé approximation problem: given $r$ Markov functions (i.e., Cauchy/Stieltjes transforms of the measures $\mu_j$)
$$
f_j(z) = \int_\R \frac{d\mu_j(x)}{z-x},
$$
find a polynomial $P_{\bm{n}}$ of degree $\le |\bm{n}|$ and polynomials $Q_{\bm{n},j}$ of degrees $\le |\bm{n}|-1$ ($j=1,\ldots,r$) such that
\begin{equation}
    P_{\bm{n}}(z) f_j(z) - Q_{\bm{n},j}(z) =\mathcal{O}(z^{-n_j-1}), \qquad z\to\infty,
\end{equation}
simultaneously for each $j=1,\ldots,r$. For more details, see, e.g.,~\cites{bookNS,VAPade}.

Today, the theory of multiple orthogonal polynomials on the real line is a very well-developed subject with applications not only in approximation theory but also in spectral theory, random matrix theory, and integrable probability (see \cite{Applications} for a short introduction, and \cite{Aptekarev,Ismail} for more comprehensive treatments). One of the fundamental questions in this theory is whether the  {\it monic} polynomial satisfying~\eqref{eq:MOPRL} is unique. In that case, its degree equals $|\bm{n}|$, and the index $\bm{n}$ is said to be normal. Unlike the case $r=1$, normality does not hold in general. 

Nevertheless, there are two broad classes of systems 
for which every multi-index $\bm{n}\in\N^r$ is normal. 
The first class consists of systems of measures whose supports lie on pairwise disjoint intervals. These are called Angelesco systems~\cite{Angelesco}. The second class consists of systems of measures  known as  AT (algebraic Chebyshev) systems, in which the measures are absolutely continuous with respect to some fixed measure $\mu$ on $\R$ and whose Radon--Nikodym derivatives satisfy certain determinantal conditions (see, e.g.,~\cite[Sect 23.1.2]{Ismail} for further details). The vast majority of explicit systems of multiple orthogonal polynomials that are studied in the literature belong to these two classes (which also include Nikishin systems). 

\smallskip

In this paper we study multiple orthogonality on the complex unit circle.  
Let $\bm{\mu} = (\mu_1,\ldots,\mu_r)$ be a system of probability measures on $\d\D = \{z\in \C: |z|=1\}$ and $\bm{n} = (n_1,\dots,n_r) \in \N^r$. 
Multiple orthogonal polynomials on the unit circle were introduced by M\'{i}nguez and Van Assche in~\cite{MOPUC1} as follows. 


\begin{defn}\label{def:normal A}
We say that a multi-index $\bm{n}$ is $\Phi$-normal with respect to $\bm{\mu}$ if there exists a unique monic polynomial $\Phi_{\bm{n}}$ (called the type II multiple orthogonal polynomial) 
of the form
\begin{equation}\label{eq:monic polyomials def A}
    \Phi_{\bm{n}}(z) = z^{\abs{\bm{n}}} + k_{\abs{\bm{n}}-1}z^{\abs{\bm{n}}-1}+\ldots+k_{0},
\end{equation}
which satisfies the orthogonality relations
\begin{equation}\label{eq:type II def A}
    \int_\T \Phi_{\bm{n}}(z)z^{-k}d\mu_j(z) = 0, \qquad k = 0,1,\ldots,n_j-1, \qquad j = 1,\dots,r.
\end{equation}
\end{defn}

In~\cite{MOPUC1,MOPUC2,KVMOPUC}, properties of such polynomials were established, including recurrence relations, compatibility conditions, a Christoffel--Darboux formula, and the associated Riemann--Hilbert and Hermite--Pad\'{e} problems. We describe the latter, as introduced in~\cite{MOPUC1}. Given $r$ Carath\'{e}odory functions, defined by
\begin{equation}\label{eq:caratheodory function}
    F_j(z) = \int_\T \frac{w+z}{w-z}d\mu_j(w), \qquad z\in\C\setminus\T,
\end{equation}
the polynomials $\Phi_{\bm{n}}$ together with certain polynomials $\Psi_{\bm{n},j}$ of degree $\le |\bm{n}|$ ($j=1,\ldots,r$) solve 
\begin{alignat}{3}
    \label{eq:HPT1}
    &\Phi_{\bm{n}}(z) F_j(z) + \Psi_{\bm{n},j}(z) =\mathcal{O}(z^{n_j}), \qquad &&z\to 0, \\
    \label{eq:HPT2}
    &\Phi_{\bm{n}}(z) F_j(z) + \Psi_{\bm{n},j}(z) =\mathcal{O}(z^{-1}), \qquad &&z\to\infty,
\end{alignat}
for each $j=1,\ldots,r$. This approximation problem is called {\it  two-point}, as it involves approximating at $0$ and at $\infty$ simultaneously. 

One of the main limitations of
the theory was   
the lack of examples of $\Phi$-perfect systems, i.e., 
systems for which every $\bm{n} \in \N^r$ is $\Phi$-normal.
To the best of our knowledge,  the only such example 
was given in~\cite[Sect. 3]{MOPUC2}, namely $(\mu_1,\mu_2)$, where $\mu_1$ is the orthogonality measure of the Rogers--Szeg\H{o} polynomials and $\mu_2$ is the Lebesgue measure.

In pursuit of more examples, we propose to work with Laurent orthogonality. This leads to the following  definition, for which we are able to generate a large class of systems where we have normality at every location.



\begin{defn}\label{def:normal B}
We say that a multi-index $\bm{n}$ is $\phi$-normal with respect to $\bm{\mu}$ if
there exists a unique function of the form
\begin{equation}\label{eq:monic type II mod}
        \phi_{\bm{n}}(z) = z^{\abs{\bm{n}}/2} + \kappa_{\abs{\bm{n}}/2-1}z^{\abs{\bm{n}}/2-1} + \ldots + \kappa_{-\abs{\bm{n}}/2}z^{-\abs{\bm{n}}/2},
\end{equation} 
which satisfies the orthogonality relations
\begin{equation}\label{eq:type II def mod}
    \int \phi_{\bm{n}}(z)z^{-k}d\mu_j(z) = 0, \quad k = -n_j/2,-n_j/2+1,\ldots,n_j/2-1.
\end{equation}
\end{defn}
Note that if $r=1$, then $\phi_n(z)$ is simply $z^{-n/2} \Phi_n(z)$, so that $\phi_n$ and $\Phi_n$ carry the same information. If $r\ge 2$, the relationship between $\phi_{\bm{n}}$ and $\Phi_{\bm{n}}$ becomes highly non-trivial.

The main result of our paper 
is 
that every multi-index $\bm{n} \in \N^r$ is {$\phi$-}normal for two wide classes of systems: Angelesco systems, where the supports of the measures are pair-wise disjoint (see Section~\ref{ss:Angelesco}), and AT systems (see Section~\ref{ss:AT}). 
In a companion paper~\cite{KNik}, we introduce Nikishin systems on the unit circle and show that they satisfy the AT property whenever  $\bm{n}$ has same-parity components with $n_1\ge n_2\ge\cdots \ge n_r$ or with $r=2$.

Just like $\Phi_{\bm{n}}$, the function $\phi_{\bm{n}}$ appears naturally as a solution to a two-point Hermite--Pad\'{e} problem. For the case of $n_j$'s even, it takes the form
\begin{alignat}{3}
\label{eq:HPT3}
    & \phi_{\bm{n}}(z)F_j(z) + \psi_{\bm{n},j}(z) = \mathcal{O}(z^{n_j/2}), \qquad && z \rightarrow 0, \\\label{eq:HPT4}
    & \phi_{\bm{n}}(z)F_j(z) + \psi_{\bm{n},j}(z) = \mathcal{O}(z^{-n_j/2-1}), \qquad && z \rightarrow \infty,
\end{alignat}
(compare with~\eqref{eq:HPT1}--\eqref{eq:HPT2}). In particular, if $\bm{\mu}$ is any Angelesco or AT system, then we show that the solution of this problem is unique, up to a multiplicative normalization. 

It can be argued that the functions $\phi_{2\bm{n}}$ are particularly well-behaving. Since every component $n_j$ is even, they are all Laurent polynomials, and the orthogonality conditions are all  with respect to integer powers. The Hermite--Pad\'{e} problem also takes a simpler form in this case. 
For the AT systems, however, it seems more natural to consider $\phi$'s at the locations 
with all odd components $n_j$, as such multi-indices correspond to the periodic AT property, see~\cite{KreNud}. It is therefore particularly striking that any multi-index, with an arbitrary combination of even and odd components, is $\phi$-normal for any Angelesco and any AT system.

In subsequent work~\cite{KVHP}, we study a related, though distinct from~\eqref{eq:monic type II mod}--\eqref{eq:type II def mod}, setting of Laurent multiple orthogonality on the unit circle. In that framework, we establish Szeg\H{o}-type recurrence relations and compatibility conditions for the associated recurrence coefficients, derive Christoffel--Darboux formulas and Heine-type determinantal representations, and ultimately obtain extensions of the Szeg\H{o} mapping and the Geronimus relations. Independently of our work, a recent manuscript~\cite{HueMan} investigates Laurent multiple orthogonality from yet another perspective. Within a broader framework of mixed orthogonality, the authors analyze recurrence relations, Christoffel--Darboux formulas, and spectral transformations. 


The organization of our paper is as follows. In Section~\ref{ss:normality} we define the type I and type II multiple orthogonal functions $\phi_{\bm{n}}$, $\bm{\xi}_{\bm{n}}$, along with their ``reversed'' versions $\phi^\sharp_{\bm{n}}$, $\bm{\xi}^\sharp_{\bm{n}}$. Section~\ref{ss:andreief}
contains some technical determinantal results, featuring a generalization of the Andreief identity, which is our main tool for showing normality.  In Section~\ref{ss:Angelesco} we define Angelesco systems on the unit circle and prove their perfectness, while in Section~\ref{ss:AT} we do the same for the AT systems on the unit circle. In Section~\ref{ss:HPphi} we describe the two-point Hermite--Pad\'{e} problem that is equivalent to the type II $\phi$-orthogonality, and in Section~\ref{ss:HPtypeI} we do the same for type I.

\smallskip

\textbf{Acknowledgments}: The authors are grateful to M.~Ma\~{n}as, A.~Foulqui\'{e}-Moreno, A.~Branquinho, W.~Van Assche, and R.~Cruz-Barroso for useful questions and discussions during the OPSFA-17 meeting in Granada (Spain). 




\section{Normality and Moment Matrices}\label{ss:normality}

We work with a fixed branch of the square root function 
\begin{equation}\label{eq:sqrt}
z^{1/2}=|z|^{1/2} \exp({i \arg_{[t_0,t_0+2\pi)}} z/2)
\end{equation}
for some chosen $t_0\in\R$. {We stress that a choice of the square root function impacts the inner products~\eqref{eq:type II def mod}  and the functions~\eqref{eq:monic type II mod} (unless $\bm{n}$ has all $n_j$'s even). As a result, the notion of $\phi$-normality in Definition~\ref{def:normal B} depends on this choice of the square root function.}

From now on, unless stated otherwise, we will use the term {\it normal} 
to refer to {\it $\phi$-normal} 
(see Definition~\ref{def:normal B}).

The index $\bm{0}$ is considered to be normal  for any system $\bm{\mu}$, with $\phi_{\bm{0}} = 1$. For $\bm{n} \neq \bm{0}$, let us introduce the $|\bm{n}|\times |\bm{n}|$
matrix 
    \begin{equation}\label{eq:moment matrix mod}
        M_{\bm{n}} =
        \left(\begin{array}{c}
            T^{(1)}_{n_1,|\bm{n}|} 
            \\[0.4em]
        \hdashline[0.5pt/2pt]
        \\[-1.2em]
         \vdots \\[0.2em]
        \hdashline[0.5pt/2pt]
        \\[-1.0em]
            T^{(r)}_{n_r,|\bm{n}|}
        \end{array}
        \right),
    \end{equation}
    where $T^{(j)}_{n_j,|\bm{n}|}$ is the $n_j\times |\bm{n}|$ matrix given by
    \begin{equation}\label{eq:H1}
         T^{(j)}_{n_j,|\bm{n}|} =
         \begin{pmatrix}
            \int z^{-\abs{\bm{n}}/2}z^{n_j/2} d\mu_j(z) & \cdots & \int z^{\abs{\bm{n}}/2-1}z^{n_j/2} d\mu_j(z) \\
            \int z^{-\abs{\bm{n}}/2}z^{n_j/2-1} d\mu_j(z) & \cdots & \int z^{\abs{\bm{n}}/2-1}z^{n_j/2-1} d\mu_j(z) \\
            \vdots & \ddots & \vdots \\ 
            \int z^{-\abs{\bm{n}}/2}z^{-n_j/2+1} d\mu_j(z) & \cdots & \int z^{\abs{\bm{n}}/2-1}z^{-n_j/2+1} d\mu_j(z) \\
        \end{pmatrix}
    \end{equation}
    


The linear system of equations \eqref{eq:type II def mod}, where we treat the $\kappa$-coefficients as unknowns, has the $|\bm{n}|\times |\bm{n}|$ coefficient matrix equal to $M_{\bm{n}}$ introduced above. Hence we get the following result.

    \begin{prop}\label{normality condition}
        $\bm{n} \neq \bm{0}$ is normal if and only if $\det M_{\bm{n}} \neq 0$.
    \end{prop}

\begin{rem}
    One may want to work with $z^{1/2}\phi_{\bm{n}}(z)$ instead of $\phi_{\bm{n}}(z)$. When $\abs{\bm{n}}$ is odd, we would then always get Laurent polynomials. The orthogonality conditions lead to the coefficient matrix $M_{\bm{n}}$ whose $j$-th block is given by
    \begin{equation}\label{eq:H1 odd}
         \begin{pmatrix}
            \int z^{-(\abs{\bm{n}}-1)/2}z^{(n_j-1)/2} d\mu_j(z) & \cdots & \int z^{(\abs{\bm{n}}-1)/2}z^{(n_j-1)/2} d\mu_j(z) \\
            \int z^{-(\abs{\bm{n}}-1)/2}z^{(n_j-3)/2} d\mu_j(z) & \cdots & \int z^{(\abs{\bm{n}}-1)/2}z^{(n_j-3)/2} d\mu_j(z) \\
            \vdots & \ddots & \vdots \\ 
            \int z^{-(\abs{\bm{n}}-1)/2}z^{-(n_j-1)/2} d\mu_j(z) & \cdots & \int z^{(\abs{\bm{n}}-1)/2}z^{-(n_j-1)/2} d\mu_j(z) \\
        \end{pmatrix}.
    \end{equation}
    This is of course equal to $T^{(j)}_{n_j,\abs{\bm{n}}}$ of \eqref{eq:H1}. We will use this form of $T^{(j)}_{n_j,\abs{\bm{n}}}$ when we compute $\det{M_{\bm{n}}}$ in the next sections. 
\end{rem}
\begin{rem}\label{rem:Phi-vs-phi}
It is easy to see that the following relationship holds between the two notions of normality from Definitions~\ref{def:normal A} and~\ref{def:normal B}. $\bm{n}$ is $\Phi$-normal for $\bm{\mu}=(\mu_j)_{j=1}^r$ if and only if $\bm{n}$ is $\phi$-normal  for $\widehat{\bm{\mu}}=(\widehat\mu_j)_{j=1}^r$ with $d\widehat\mu_j(z)=z^{(n_j-|\bm{n}|)/2}d\mu_j(
z)$. In this case $\Phi_{\bm{n}}(z)$ with respect to $\bm{\mu}$ coincides with $z^{|\bm{n}|/2} {\phi}_{\bm{n}}(z)$ with respect to $\widehat{\bm{\mu}}$.

 \end{rem}

Similarly to multiple orthogonal polynomials on the real line, we can easily obtain the determinantal formula
    \begin{equation}\label{eq:HeineII}
        \phi_{\bm{n}}(z) = \frac{1}{\det M_{\bm{n}}}
        \det
        \left(\begin{array}{ccc}
            &\tilde{T}^{(1)}_{n_1,|\bm{n}|} 
            \\[0.4em]
        \hdashline[0.5pt/2pt]
        \\[-1.2em]
        & \vdots \\[0.2em]
        \hdashline[0.5pt/2pt]
        \\[-1.0em]
            &\tilde{T}^{(r)}_{n_r,|\bm{n}|} 
            \\[0.4em]
        \hdashline[0.5pt/2pt]
        \\[-1.0em]
            z^{-\abs{\bm{n}}/2} & \cdots & z^{\abs{\bm{n}}/2}
        \end{array}\right),
    \end{equation}
    where $\tilde{T}^{(j)}_{n_1,|\bm{n}|}$ is {the $n_j\times (|\bm{n}|+1)$ matrix} given by
    \begin{equation}\label{eq:H1 tilde}
         \tilde{T}^{(j)}_{n_j,|\bm{n}|} = \begin{pmatrix}
            \int z^{-\abs{\bm{n}}/2}z^{n_j/2} d\mu_j(z) & \cdots & \int z^{\abs{\bm{n}}/2}z^{n_j/2} d\mu_j(z) \\
            \int z^{-\abs{\bm{n}}/2}z^{n_j/2-1} d\mu_j(z) & \cdots & \int z^{\abs{\bm{n}}/2}z^{n_j/2-1} d\mu_j(z) \\
            \vdots & \ddots & \vdots \\ 
            \int z^{-\abs{\bm{n}}/2}z^{-n_j/2+1} d\mu_j(z) & \cdots & \int z^{\abs{\bm{n}}/2}z^{-n_j/2+1} d\mu_j(z) \\
        \end{pmatrix}.
    \end{equation}
Indeed, it is easy to see that the right-hand side of \eqref{eq:HeineII} {has the form~\eqref{eq:monic type II mod}}, and satisfies all the orthogonality conditions {\eqref{eq:type II def mod}.} 

We will write $f^\sharp$ for the function 
\[
f^\sharp(z) = \overline{f(1/\bar{z})}.
\]
Then $\phi_{\bm{n}}^\sharp$
has the form
\begin{equation}\label{eq:monic type II mod sharp}
        \phi^\sharp_{\bm{n}}(z) = \bar\kappa_{-\abs{\bm{n}}/2}z^{\abs{\bm{n}}/2} + \ldots+ \bar\kappa_{\abs{\bm{n}}/2-1}z^{-\abs{\bm{n}}/2+1} + z^{-\abs{\bm{n}}/2} ,
\end{equation} 
and satisfies the orthogonality relations
\begin{equation}\label{eq:type II def mod sharp}
    \int \phi^\sharp_{\bm{n}}(z)z^{-k}d\mu_j(z) = 0, \qquad k = -n_j/2+1,-n_j/2+2,\ldots,n_j/2.
\end{equation}
It is clear that normality of $\bm{n}$ is equivalent to the uniqueness and existence of $\phi_{\bm{n}}^\sharp$ with such properties; a quick check shows that the system of orthogonality relations for $\phi_{\bm{n}}^\sharp$ has the coefficient matrix equal to $M_{\bm{n}}$ defined above. Similarly to \eqref{eq:HeineII}, one gets 
    \begin{equation}\label{eq:HeineII star}
        \phi_{\bm{n}}^\sharp(z) = \frac{1}{\det \overline{M_{\bm{n}}}}
        \det
        \left(\begin{array}{ccc}
            &\overline{\tilde{T}^{(1)}_{n_1,|\bm{n}|}} \\[0.4em]
        \hdashline[0.5pt/2pt]
        \\[-1.2em]
        & \vdots \\[0.2em]
        \hdashline[0.5pt/2pt]
        \\[-1.0em]
            &\overline{\tilde{T}^{(r)}_{n_r,|\bm{n}|}} 
             \\[0.4em]
        \hdashline[0.5pt/2pt]
        \\[-1.0em]
            z^{\abs{\bm{n}}/2} & \cdots & z^{-\abs{\bm{n}}/2}
        \end{array}\right).
    \end{equation}

    We also want to introduce type I orthogonality. Here we use vectors of functions $\bm{\xi}_{\bm{n}} = (\xi_{\bm{n},1},\dots,\xi_{\bm{n},1})$, defined through the following proposition. 

    \begin{prop}\label{prop:typeI}
    A multi-index $\bm{n} \in \N^r$ is normal if and only if there exists a unique vector of functions $\bm{\xi}_{\bm{n}} = (\xi_{\bm{n},1},\dots,\xi_{\bm{n},r})$, with $\xi_{\bm{n},j}$ of the form
    \begin{equation}\label{eq:type I coefficients}
       \xi_{\bm{n},j}(z) = \lambda_{-n_j/2,j}z^{-n_j/2} +\lambda_{-n_j/2+1,j}z^{-n_j/2+1} + \dots + \lambda_{n_j/2-1,j}z^{n_j/2-1},
    \end{equation}
    satisfying the orthogonality relations 
    \begin{equation}\label{eq:type I def SWAPPED}
        \sum_{j = 1}^r \int \xi_{\bm{n},j}(z)z^{-k}d\mu_j(z) = 
        \begin{cases}
            0, & k = -\tfrac{\abs{\bm{n}}}{2}+1,\dots,\tfrac{\abs{\bm{n}}}{2}-1, \\
            1, & k = -\tfrac{\abs{\bm{n}}}{2}.
        \end{cases}
    \end{equation}
    
    Similarly, $\bm{n}$ is normal if and only if $\xi_{\bm{n}}^\sharp$ is uniquely determined by requiring 
    \begin{equation}\label{eq:type I star def SWAPPED}
        \xi^\sharp_{\bm{n},j}(z) = \kappa_{n_j/2,j}z^{n_j/2} + \kappa_{n_j/2-1,j}z^{n_j/2-1} \dots + \kappa_{-n_j/2+1,j}z^{-n_j/2+1},
    \end{equation}
    and the orthogonality relations
    \begin{equation}\label{eq:type I star def SWAPPED}
        \sum_{j = 1}^r \int \xi^\sharp_{\bm{n},j}(z)z^{-k}d\mu_j(z) = 
        \begin{cases}
            0, & k = -\tfrac{\abs{\bm{n}}}{2}+1,\dots,\tfrac{\abs{\bm{n}}}{2}-1, \\
            1, & k = \tfrac{\abs{\bm{n}}}{2}.
        \end{cases}
    \end{equation}
\end{prop}
\begin{rem}
If $r=1$ then $\phi_{n}(z) = z^{-n/2} \Phi_n(z)$, $\phi_n^\sharp(z) = z^{-n/2} \Phi_n^*(z)$, 
while $\xi_n(z) = {\kappa}_{n-1}^{-1} z^{-n/2} \Phi^*_{n-1}(z)$, $\xi^\sharp_n(z) = \kappa_{n-1}^{-1} z^{-n/2+1} \Phi_{n-1}(z)$,  with $\kappa_{n-1} = \int \Phi_{n-1}(z) z^{-(n-1)}d\mu(z) = ||\Phi_{n-1}||^2$. Here $\Phi_n^*(z) = z^n \overline{\Phi_n(1/\bar{z})}$ as usual in the classical one-measure case.
\end{rem}
\begin{proof}
    Solving for the coefficients of $\xi_{\bm{n},1},\dots,\xi_{\bm{n},r}$, it is easy to verify that \eqref{eq:type I star def SWAPPED} is a linear system of equations with coefficient matrix equal to ${M}_{\bm{n}}^T$ (for an appropriate order of the equations and coefficients). The first statement then follows from Proposition \ref{normality condition}. Similarly, \eqref{eq:type I def SWAPPED} gives a system of equations with the same matrix ${M}_{\bm{n}}^T$, which proves the second statement. 
\end{proof}



Just as for the multiple orthogonal polynomials on the real line, see~\cite{Kui}, we can get the determinantal formula for $\xi_{\bm{n},j}$ and $\xi^\sharp_{\bm{n},j}$.   

\section{Two Determinantal Identities}\label{ss:andreief}
In our proof of normality for Angelesco and AT systems in Sections~\ref{ss:Angelesco} and \ref{ss:AT}, we rely on the following determinantal formula, which can be viewed as a generalization of the Andreief (or continuous Cauchy--Binet) identity. Its proof follows similar lines to that of the standard Andreief identity. The block form of the formula is particularly well suited for the multiple orthogonality setting, whether on the unit circle or the real line, as further demonstrated in~\cite{KNik,KVNikInt}. See also~\cite{CouVA,Kui} for the related normality proofs on the real line. 
\begin{prop}
    Let $f_j,g_j \in L^2(\mu)$ for some probability measure on a measure space $(X,\Sigma,\mu)$.  Then for any $N\ge M\ge 1$  and any $(N-M)\times N$ matrix $\bm{A}$:
    \begin{multline}\label{eq:Andreief}
        \det\left(
        \begin{array}{@{}c@{}}
        \begin{pmatrix}
            \int_X f_j(x) g_k(x)\,d\mu(x)
        \end{pmatrix}_{1\le j \le M, 1\le k\le N} 
        \\[0.4em]
        \hdashline[0.5pt/2pt]
        \\[-1.0em]
        \bm{A}
        \end{array}
        \right) =
        \\
        \tfrac{1}{M!}
         \int_{X^M}
        \det\left(
        \begin{array}{@{}c@{}}
        \begin{pmatrix}
            g_k(x_j)
        \end{pmatrix}_{1\le j \le M, 1\le k\le N} 
        \\[0.4em]
        \hdashline[0.5pt/2pt]
        \\[-1.0em]
        \bm{A}
        \end{array}
        \right)
        \det\left(f_l(x_j)\right)_{1\le l,j \le M} \,d^M\mu(\bm{x}),
    \end{multline}
    where $d^M\mu(\bm{x}):=d\mu(x_1)\ldots d\mu(x_M)$. The matrix on the left-hand side refers to the matrix whose upper
    $M\times N$ block is $\begin{pmatrix}
            \int_X f_j(x) g_k(x)\,d\mu(x)
        \end{pmatrix}_{1\le j \le M, 1\le k\le N}$ and whose lower $(N-M)\times N$ block is $\bm{A}$.
\end{prop}
\begin{proof}
    Expand the right hand side of \eqref{eq:Andreief} using the Leibniz formula to get 
    \begin{equation}\label{eq:andreief proof step 1}
        \sum_{\sigma \in S_N, \tau \in S_M} \sgn{(\sigma)} \sgn{(\tau)} \prod_{j = M+1}^N A_{j,\sigma(j)} \int \prod_{j = 1}^M g_{\sigma(j)}(x_j)f_j(x_{\tau(j)}) d^M\mu(\bm{x}).
    \end{equation}
    After some rearrangement of the factors in the above integral we can write 
    \begin{equation}\label{eq:andreief proof step 2}
        \sum_{\sigma \in S_N, \tau \in S_M} \sgn{(\sigma)} \sgn{(\tau)} \prod_{j = M+1}^N A_{j,\sigma(j)} \prod_{j = 1}^M \int g_{\sigma(j)}(x)f_{\tau^{-1}(j)}(x) d\mu(x).
    \end{equation}
    Write $S_{N|M}$ for the set of permutations in $S_N$ that fix $M+1,\dots,N$. Associate each $\tau \in S_M$ with $\tilde{\tau} \in S_{N|M}$, so that we can rewrite \eqref{eq:andreief proof step 2} as
    \begin{equation}\label{eq:andreief proof step 3}
        \sum_{\sigma \in S_N, \tilde{\tau} \in S_{N|M}} \sgn{(\sigma\tilde{\tau})} \prod_{j = M+1}^N A_{j,\sigma\tilde{\tau}(j)} \prod_{j = 1}^M \int g_{\sigma\tilde{\tau}(j)}(x)f_j(x) d\mu(x).
    \end{equation}
    Given any $\tilde{\sigma} \in S_N$, for each $\tilde{\tau} \in S_{N|M}$ there is a unique $\sigma \in S_N$ such that $\tilde{\sigma} = \sigma\tilde{\tau}$. Hence the above sum runs through each $\tilde{\sigma} \in S_M$ exactly $M!$ times, and we end up with
    \begin{equation}\label{eq:andreief proof step 4}
        M!\sum_{\tilde{\sigma} \in S_N} \sgn{(\tilde{\sigma})} \prod_{j = M+1}^N A_{j,\tilde{\sigma}(j)} \prod_{j = 1}^M \int g_{\tilde{\sigma}(j)}(x)f_j(x) d\mu(x),
    \end{equation}
    which is exactly $M!$ times the left hand side of \eqref{eq:Andreief} expanded using the Leibniz formula.
\end{proof}

We will also need the following Vandermonde-type determinant. 

\begin{lem}  
Let
\begin{equation}
    \label{eq:vandEven}
        V_N(\bm{z}):=   \det \begin{pmatrix}
                z_1^{-(N-1)/2} & z_1^{-(N-3)/2} & \cdots &  z_1^{(N-1)/2} \\
                z_2^{-(N-1)/2}  & z_2^{-(N-3)/2} & \cdots & z_2^{(N-1)/2} \\
                \vdots & \vdots &  \ddots & \vdots\\
                z_{N}^{-(N-1)/2}  & z_{N}^{-(N-3)/2} & \cdots & z_{N}^{(N-1)/2}      
            \end{pmatrix},
\end{equation}
where $\bm{z} = (z_1,\dots,z_N)$.

If $z_j = e^{i\theta_j}\in\T$ for all $j = 1,\dots,N$, and 
        $t_0\le \theta_1\le\theta_2\le \ldots\le \theta_N\le t_0+2\pi$, then
        \begin{align}
\label{eq:vandEven2}
         V_N(\bm{z})   & = 
             (2i)^{N(N-1)/2}
\prod_{1 \leq j < k \leq N}\sin\tfrac{\theta_k-\theta_j}{2}
            \\ \label{eq:vandEven3}
            & =
             i^{N(N-1)/2}
\prod_{1 \leq j < k \leq N}|z_k-z_j|.
        \end{align} 
\end{lem}
\begin{proof}
     By extracting $z_j^{-(N-1)/2}$ from the $j$-th row we obtain the Vandermonde determinant. Therefore~\eqref{eq:vandEven} is equal to 
    \[
    \prod_{j=1}^{N} z_j^{-(N-1)/2} \prod_{1 \leq j<k \leq N}  (z_k-z_j), 
    \]
    which is equal to~\eqref{eq:vandEven2} after we employ the elementary identity
    \begin{equation}\label{eq:zSin}
        e^{i\theta_k}-e^{i\theta_j} = 2i e^{i(\theta_j+\theta_k)/2} \sin\tfrac{\theta_k-\theta_j}{2}.
    \end{equation}
    Finally,~\eqref{eq:vandEven3} is clear if one observes that the product of sines in~\eqref{eq:vandEven2} is always non-negative because of the ordering of $\theta_j$'s.
\end{proof}


\section{Angelesco systems on the unit circle}\label{ss:Angelesco}
\begin{defn}\label{def:Angelesco}
    For each $j = 1,\dots,r$, let $\mu_j$ be a probability measure on $\T$ with infinite support. We say that system $\bm{\mu} = (\mu_1,\dots,\mu_r)$ is an Angelesco system on the unit circle if there exist closed arcs $\Gamma_j\subset \T$ such that $\supp\,\mu_j\subseteq \Gamma_j$, $j = 1,\dots,r$, and $|\Gamma_j\cap \Gamma_k|\in\{0,1,2\}$ whenever $j \neq k$.
\end{defn}
\begin{rem}
    In the above definition, it is easy to see that $\Gamma_j\cap \Gamma_k$ can consist of two points only if $r=2$ and the union of $\Gamma_1$ and $\Gamma_2$ is equal to $\T$.
\end{rem}


{We assume that the order of $\Gamma_j$'s is chosen so that $t_0 \le  \theta_1 \le \dots \le \theta_r \le t_0 + 2\pi$ holds whenever $e^{i\theta_1} \in \Gamma_1,\dots,e^{i\theta_r} \in \Gamma_r$ for some fixed $t_0\in\R$. We choose the branch of the square root function to be~\eqref{eq:sqrt} with this  $t_0$. 
Let us additionally require that $\mu_r$ does not have a point mass at $e^{it_0}$ (see, however, Remark~\ref{rem:badAngelesco} below).}




\begin{thm}\label{thm:Angelesco}
    For Angelesco systems on the unit circle, every $\bm{n} \in \N^r$ is normal.
\end{thm}
\begin{proof}
We show that $\det M_{\bm{n}} \ne 0$ by using the generalized Andreief identity~\eqref{eq:Andreief}, applied $r$ times to each of the $T^{(l)}_{n_l,|\bm{n}|}$, $l=1,\ldots,r$, sequentially (with $N=|\bm{n}|$, $M=n_l$, with $f_j(z)=z^{(n_l+1)/2-j}$, $g_k(z)=z^{-(|\bm{n}|+1)/2+k}$, see~\eqref{eq:H1 odd}) to obtain

\begin{equation}\label{eq:angelesco determinant}
    \det M_{\bm{n}} = \frac{(-1)^{k_{\bm{n}}}}{n_1!\ldots n_r!}\int_{\Gamma_1^{n_1}}\ldots  \int_{\Gamma_r^{n_r}}
    V_{|\bm{n}|}(\bm{z}_1,\ldots,\bm{z}_r)
    \prod_{j=1}^r V_{n_j}(\bm{z}_j)
    \prod_{j=1}^r d^{n_j} \mu_j(\bm{z}_j),
\end{equation}
for an integer $k_{\bm{n}}$, where $V_{\abs{\bm{n}}}$ and $V_{n_1},\dots,V_{n_r}$ are given by \eqref{eq:vandEven}. The factor $(-1)^{k_{\bm{n}}}$ enters because the powers in $V_{n_1},\dots,V_{n_r}$ are in the reverse order when we apply the Andreief identity. Hence we have 
\begin{equation}\label{eq:kn}
    k_{\bm{n}} = \sum_{j = 1}^r n_j(n_j-1)/2.
\end{equation}

For some choice of permutations $\sigma_1 \in S_{n_1},\dots, \sigma_r \in S_{n_r}$, assume that $t_0 \le \theta_{j,\sigma(1)} < \dots < \theta_{j,\sigma(n_j)} < t_0 + 2\pi$, where $\bm{z}_j = (e^{i\theta_{j,1}},\dots,e^{i\theta_{j,n_j}})$, and write $\bm{z}_{j,\sigma_j}$ for $(z_{j,\sigma_j(1)},\dots,z_{j,\sigma_j(n_j)})$, $j = 1,\dots,r$. Then we have 
\begin{align}\label{eq:big vandermonde permutation}
    & V_{\abs{\bm{n}}}(\bm{z}_1,\dots,\bm{z}_r) = V_{\abs{\bm{n}}}(\bm{z}_{1,\sigma_1},\dots,\bm{z}_{r,\sigma_r})\prod_{j = 1}^r\sgn(\sigma_j), \\ \label{eq:small vandermonde permutation}
    & V_{n_j}(\bm{z}_j) = V_{n_j}(\bm{z}_{j,\sigma_j})\sgn(\sigma_j), \qquad j = 1,\dots,r.
\end{align}
Since we are in an Angelesco system, the elements in $(\bm{z}_{1,\sigma_1},\dots,\bm{z}_{r,\sigma_r})$ are now ordered by size of angle in $[t_0,t_0+2\pi)$, so that \eqref{eq:vandEven3} can be applied. 

Consider the regions 
\begin{multline}
    \Gamma_{j,\sigma_j} = \set{(e^{i\theta_{j,\sigma(1)}},\dots,e^{i\theta_{j,\sigma(n_j)}}) \in \Gamma_j^{n_j}: \theta_{j,\sigma(1)} < \dots < \theta_{j,\sigma(n_j)}}, \\ j = 1,\dots,r. 
\end{multline}
By \eqref{eq:vandEven3} and \eqref{eq:big vandermonde permutation}-\eqref{eq:small vandermonde permutation}, the integral in \eqref{eq:angelesco determinant} over the region $\prod_{j=1}^r\Gamma_{j,\sigma}^{n_j}$ turns into
\begin{equation}\label{eq:int}
    i^{l_{\bm{n}}} \int_{\Gamma_{1,\sigma_1}}\ldots  \int_{\Gamma_{r,\sigma_r}}
    \prod_{j<k} |\Delta(\bm{z}_{k,\sigma_k};\bm{z}_{j,\sigma_j})|
    \prod_{j=1}^r |\Delta(\bm{z}_{j,\sigma_j})|^2
    \prod_{j=1}^r d^{n_j} \mu_j(\bm{z}_j),
\end{equation}
where $\Delta(\bm{z}) := \prod_{j<k} (z_k-z_j)$ and $\Delta(\bm{z};\bm{y}):=\prod_{j,k}(z_k-y_j)$, and 
\begin{equation}\label{eq:ln}
    l_{\bm{n}} = \abs{\bm{n}}(\abs{\bm{n}}-1)/2 + \sum_{j = 1}^r n_j(n_j-1)/2.
\end{equation}
Hence we now have a {non-negative} integrand. Because each $\mu_j$ is supported on an infinite number of points, the integral in~\eqref{eq:int} is, in fact, positive.
Finally, the summation over all different permutations shows that $\det M_{\bm{n}}\ne 0$, i.e, $\bm{n}$ is normal. 
\end{proof}

\begin{rem}
    Note that the integral formula \eqref{eq:angelesco determinant} works for any system of measures.
\end{rem}
\begin{rem}\label{rem:badAngelesco}
    {One can adapt the proof above to the excluded case when $\mu_r$ has a pure point at $e^{it_0}$. For this, one needs to choose the square root function in the integrals with respect to  $\mu_r$ to be taken with the complex argument $\operatorname{arg}_{(t_0,t_0+2\pi]}$, while for all the other measures we take $\operatorname{arg}_{[t_0,t_0+2\pi)}$ as before. 
    }
\end{rem}

\section{AT Systems on the Unit Circle}\label{ss:AT}
    Recall that a collection $\set{u_j(t)}_{j=1}^n$ of continuous real-valued functions on a closed interval $[a,b]$ is called a Chebyshev system on $[a,b]$ if the determinants
    \begin{equation}\label{eq:Chebyshev}
        W_n(\bm{x}):=\det
        \begin{pmatrix}
            u_1(x_1) & u_1(x_2) & \cdots & u_1(x_n) \\
            u_2(x_1) & u_2(x_2) & \cdots & u_2(x_n) \\
            \vdots & \vdots & \ddots & \vdots \\
            u_n(x_1) & u_n(x_2) & \cdots & u_n(x_n) \\
        \end{pmatrix} \ne 0
    \end{equation}
    for every choice of different points $x_1,\ldots,x_n\in[a,b]$. In fact, by continuity, the sign of the determinant in~\eqref{eq:Chebyshev} is always positive or always negative, if we additionally introduce the ordering $x_1<x_2<\ldots<x_n$.

    Given a function $f(\theta)$ let us introduce the following notation:
    \begin{equation}
        \operatorname{Trig}_m(f) = \begin{cases}
            \set{f(\theta)\cos{\frac{2k-1}{2}\theta},f(\theta)\sin{\frac{2k-1}{2}\theta}}_{k = 1}^{m/2}, & \mbox{ if } m \mbox{ is even},
            \\
            \set{1}\cup\set{f(\theta)\cos{k\theta},f(\theta)\sin{k\theta}}_{k = 1}^{(m-1)/2}, & \mbox{ if } m \mbox{ is odd}.
        \end{cases}
    \end{equation}

    \begin{defn}\label{def:AT}
        For each $j = 1,\dots,r$, let $d\mu_j(e^{i\theta})=w_j(\theta)d\mu(e^{i\theta})$ be a measure that is absolutely continuous with respect to a probability measure $\mu$, infinitely supported on an arc $\Gamma = \{e^{i\theta}:\alpha\le \theta\le \beta\}$ of the unit circle $\T$ with $0<\beta-\alpha\le 2\pi$. We call $\bm{\mu}$ an AT system on $\Gamma$ for the multi-index $\bm{n} = (n_1,\ldots,n_r)$ if 
        \begin{equation}\label{eq:AT}
            \operatorname{Trig}_{\bm{n}}(\bm{\mu}) = \bigcup_{j=1}^r \operatorname{Trig}_{n_j}(w_j)
        \end{equation}
        is Chebyshev on $[\alpha,\beta]$.
    \end{defn}

    Given an AT system on the unit circle, it is natural to pick the square root function in~\eqref{eq:sqrt} with $t_0 = \alpha$.
    
    \begin{thm}\label{thm:AT}
        Suppose that $\bm{\mu}$ is an AT system for the index $\bm{n}$. Then $\bm{n}$ is normal.
    \end{thm}
    \begin{proof}
        If we apply the Andreief identity \eqref{eq:Andreief} with $M=N=|\bm{n}|$, then we get
        \begin{equation}
            \det M_{\bm{n}} = \frac{1}{|\bm{n}|!} \int_{\Gamma^{|\bm{n}|}} \widetilde{W}_{\bm{n}}(\bm{z}) V_{|\bm{n}|}(\bm{z}) d^{|\bm{n}|} \mu(\bm{z}),
        \end{equation}
        where $V_{\abs{\bm{n}}}$ is the Vandermonde determinant \eqref{eq:vandEven}, and $\widetilde{W}_{{\bm{n}}}$ is the determinant
        \begin{equation}\label{eq:at determinant}
            \widetilde{W}_{{\bm{n}}}(\bm{z}) = \det
            \left(\begin{array}{ccc}
                w_1(z_1)z_1^{(n_1-1)/2} & \cdots & w_1(z_{\abs{\bm{n}}})z_{\abs{\bm{n}}}^{(n_1-1)/2} \\
                \vdots & \ddots & \vdots \\
                w_1(z_1)z_1^{-(n_1-1)/2} & \cdots & w_1(z_{\abs{\bm{n}}})z_{\abs{\bm{n}}}^{-(n_1-1)/2} 
        \\[0.4em] \hdashline[0.5pt/2pt] \\[-1.2em] &\vdots \\[0.2em] \hdashline[0.5pt/2pt] \\[-1.0em]
                w_r(z_1)z_1^{(n_r-1)/2} & \cdots & w_r(z_{\abs{\bm{n}}})z_{\abs{\bm{n}}}^{(n_r-1)/2} \\
                \vdots & \ddots & \vdots \\
                w_r(z_1)z_1^{-(n_r-1)/2} & \cdots & w_r(z_{\abs{\bm{n}}})z_{\abs{\bm{n}}}^{-(n_r-1)/2} \\
            \end{array}\right).
        \end{equation}
        By performing elementary row operations on $\widetilde{W}_{{\bm{n}}}$, we can reduce this determinant to $W_{\abs{\bm{n}}} (\theta_1,\dots,\theta_{\abs{\bm{n}}})$ given by~\eqref{eq:Chebyshev}  for $\operatorname{Trig}_{\bm{n}}(\bm{\mu})$, up to sign changes and multiplication by an integer power of $2i$.
        
        Now observe that $W_{\abs{\bm{n}}}V_{\abs{\bm{n}}}$ is invariant under permutations of $z_j$'s, while on the subset $\theta_1 < \ldots < \theta_{\abs{\bm{n}}}$ of $[\alpha,\beta]^{\abs{\bm{n}}}$, the product $i^{-\abs{\bm{n}}(\abs{\bm{n}}-1)/2}W_{\abs{\bm{n}}}V_{\abs{\bm{n}}}$ is of constant sign by the AT property. Hence we obtain $\det M_{\bm{n}}\ne 0$, so that $\bm{n}$ is normal. 
    \end{proof}


\section{A two-point Hermite--Padé problem associated with 
\texorpdfstring{$\phi_{\bm{n}}$}{phis} } \label{ss:HPphi}

\subsection{The case of all even \texorpdfstring{$n_j$'s}{n_j's}}\label{ss:HPphiEven}

\noindent

First let us consider those indices $\bm{n}\in \N^r$ that have all the components $n_j$ even, $j=1,\ldots,r$. The general case will be handled in Section~\ref{ss:HPphiFull} below.

For presentation purposes let us introduce $\bm{\ell}=(\ell_1,\ldots,\ell_r)\in\N^r$ by $\bm{n} = 2\bm{\ell}$, so that, in particular, $|\bm{\ell}|=|\bm{n}|/2$ and $\ell_j= n_j/2$.

For such $\bm{n}$, we want to solve the two-point Hermite--Padé problem of finding Laurent polynomials 
\begin{equation}\label{eq:span}
 \varphi_{\bm{n}},\psi_{\bm{n},1},\ldots,\psi_{\bm{n},r} \in\operatorname{span}\big\{z^k\big\}_{k = -|\bm{\ell}|}^{|\bm{\ell}|},   
\end{equation}
so that
\begin{alignat}{3}
\label{eq:two point hermite pade eq 1}
    & \varphi_{\bm{n}}(z)F_j(z) + \psi_{\bm{n},j}(z) = \mathcal{O}(z^{\ell_j}), \qquad && z \rightarrow 0, \\\label{eq:two point hermite pade eq 2}
    & \varphi_{\bm{n}}(z)F_j(z) + \psi_{\bm{n},j}(z) = \mathcal{O}(z^{-\ell_j-1}), \qquad &&z \rightarrow \infty,
\end{alignat}
hold simultaneously for all $j=1,\ldots,r$. 
Here $F_j$ is the Carathéodory function of $\mu_j$,  $j = 1,\ldots,r$, defined by~\eqref{eq:caratheodory function}. 
$F_j$ has the power series expansions 
\begin{alignat}{3}
\label{eq:F0}
    & F_j(z) = 1 + 2\sum_{k=1}^\infty c_{k,j} z^k, \qquad&& \abs{z}  < 1, \\
        \label{eq:Finf}
    & F_j(z) = -1 - 2\sum_{k=1}^{\infty} c_{-k,j} z^{-k}, \qquad&& \abs{z} > 1,
\end{alignat}
where $c_{k,j}$ are the moments
\begin{equation}\label{eq:moments of measures}
    c_{k,j} = \int z^{-k} d\mu_j(z), \qquad k \in \Z.
\end{equation}

\begin{thm}\label{thm:hermite pade orthogonality}
     Let $n_j$ be even for all $j$. {Any} 
    Laurent polynomial $\varphi_{\bm{n}}$ {that solves the Hermite--Pad\'{e} problem~\eqref{eq:span},~\eqref{eq:two point hermite pade eq 1},~\eqref{eq:two point hermite pade eq 2}}, satisfies the orthogonality relations
    \begin{equation}\label{eq:hermite pade orthogonality}
        \int \varphi_{\bm{n}}(z)z^{-k}d\mu_j(z) = 0, \quad k = -\ell_j,-\ell_j+1,\dots,\ell_j-1, \quad j = 1,\dots,r.
    \end{equation}
    Moreover, the Laurent polynomials $\psi_{\bm{n},j}$ can be expressed in terms of $\varphi_{\bm{n}}$ by
    \begin{equation}\label{eq:second kind polynomials}
        \psi_{\bm{n},j}(z) = \int \frac{w+z}{w-z}\big(\varphi_{\bm{n}}(w)-\varphi_{\bm{n}}(z)\big)d\mu_j(w) + \int \varphi_{\bm{n}}(w)d\mu_j(w). 
    \end{equation}

    {Conversely, any $\varphi_{\bm{n}} \in \operatorname{span}\set{z^k}_{k = -|\bm{\ell}|}^{|\bm{\ell}|}$, that satisfies the orthogonality relations~\eqref{eq:hermite pade orthogonality}, solves the Hermite--Pad\'{e} problem ~\eqref{eq:two point hermite pade eq 1}--\eqref{eq:two point hermite pade eq 2} together with~\eqref{eq:second kind polynomials}.}
\end{thm}
\begin{rem}
    Note that the last term $\int \varphi_{\bm{n}}(w)d\mu_j(w)$ in \eqref{eq:second kind polynomials} is zero if $n_j>0$, but not otherwise in general. 
\end{rem}
\begin{proof}
    We fix $j$ and write $\varphi_{\bm{n}}$ and $\psi_{\bm{n},j}$ in the form 
    \begin{align}\label{eq:coefficients of chi}
         \varphi_{\bm{n}}(z) &= \kappa_{|\bm{\ell}|}z^{|\bm{\ell}|} + \kappa_{|\bm{\ell}|-1}z^{|\bm{\ell}|-1}+\dots + \kappa_{-|\bm{\ell}|}z^{-|\bm{\ell}|}, \\ \label{eq:coefficients of y}
        \psi_{\bm{n},j}(z) & = \lambda_{|\bm{\ell}|,j}z^{|\bm{\ell}|} + \lambda_{|\bm{\ell}|-1,j}z^{|\bm{\ell}|-1}+ \dots + \lambda_{-|\bm{\ell}|,j}z^{-|\bm{\ell}|}.
    \end{align}
    We then have 
    \begin{alignat*}{3}
        & \varphi_{\bm{n}}(z)F_j(z) = a_{-|\bm{\ell}|,j}z^{-|\bm{\ell}|} + \dots + a_{\ell_j-1,j}z^{\ell_j-1} + \mathcal{O}(z^{\ell_j}),  \qquad&&z \rightarrow 0, \\
        & \varphi_{\bm{n}}(z)F_j(z) = b_{|\bm{\ell}|,j}z^{|\bm{\ell}|} + \dots + b_{-\ell_j,j}z^{-\ell_j} + \mathcal{O}(z^{-\ell_j-1}),  \qquad&&z \rightarrow \infty,
    \end{alignat*}
    where the coefficients are given by
    \begin{alignat}{3}
    \label{eq:coefficients of series 1}
         & a_{k,j} = \kappa_k + 2\kappa_{k-1}c_{1,j} + \dots + 2\kappa_{-\abs{\bm{\ell}}} c_{\abs{\bm{\ell}}+k,j},  \quad&&k = -|\bm{\ell}|,\dots,\ell_j-1, 
         \\ 
         \label{eq:coefficients of series 2}
         & b_{k,j} = -\kappa_k - 2\kappa_{k+1} c_{-1,j} - \dots - 2 \kappa_{\abs{\bm{\ell}}} c_{-\abs{\bm{\ell}}+k,j}, \quad&&k = -\ell_j,\dots,|\bm{\ell}|.
    \end{alignat}
    To get \eqref{eq:two point hermite pade eq 1}-\eqref{eq:two point hermite pade eq 2} we necessarily need
    \begin{alignat}{3}
        \label{eq:PadeToOrthoEq1}
        & -\lambda_{k,j} = a_{k,j}, \qquad&& k = -|\bm{\ell}|,\dots,-\ell_j-1, \\ 
        \label{eq:PadeToOrthoEq2}
        & -\lambda_{k,j} = a_{k,j}=b_{k,j},  \qquad&& k = -\ell_j,\dots,\ell_j-1, \\ 
        \label{eq:PadeToOrthoEq3}
        & -\lambda_{k,j} = b_{k,j},  \qquad&&k  = \ell_j,\dots,|\bm{\ell}|.
    \end{alignat}
    Observe, in particular, that $\psi_{\bm{n},j}$ is uniquely determined from ~\eqref{eq:PadeToOrthoEq1}--\eqref{eq:PadeToOrthoEq3}, assuming $\varphi_{\bm{n}}$ is given. Furthermore,~\eqref{eq:PadeToOrthoEq2}, combined with~\eqref{eq:coefficients of series 1}--\eqref{eq:coefficients of series 2} and~\eqref{eq:moments of measures}, turns into
    \begin{equation}
         \int \big(\kappa_{|\bm{\ell}|}z^{|\bm{\ell}|} + \dots + \kappa_{-|\bm{\ell}|}z^{-|\bm{\ell}|}\big)z^{-k}d\mu_j(z) = 0, \qquad k = -\ell_j,\dots,\ell_j-1.
    \end{equation}

    This proves \eqref{eq:hermite pade orthogonality}.  Now denote
    \begin{equation}\label{eq:secondKindFunction}
        R_{\bm{n},j}(z) = \int \frac{w+z}{w-z} \varphi_{\bm{n}}(w) d\mu_j(w).
    \end{equation}
    By
    \begin{equation}\label{eq:kernel}
        \frac{w+z}{w-z}
        =
        \begin{cases}
            1+ 2\sum_{k=1}^\infty \frac{z^k}{w^k},\quad  & \mbox{if } |z|<|w|, \\
            -1 - 2\sum_{k=1}^\infty \frac{w^k}{z^k}, \quad & \mbox{if } |z|>|w|, \\
        \end{cases}
    \end{equation}
    and \eqref{eq:hermite pade orthogonality}, we get, as $z\to 0$,
    \begin{equation}\label{eq:R0}
        R_{\bm{n},j}(z)
        =\int \varphi_{\bm{n}}(w)d\mu_j(w) +
        \sum_{k=1}^{\infty} 2z^k \int \varphi_{\bm{n}}(w) w^{-k}d\mu_j(w)   = 
        \mathcal{O}(z^{\ell_j}),
    \end{equation}
    and as $z\to\infty$,
    \begin{multline}\label{eq:RInfty}
        R_{\bm{n},j}(z)
        =-\int \varphi_{\bm{n}}(w)d\mu_j(w) -
        \sum_{k=1}^{\infty} 2z^{-k} \int \varphi_{\bm{n}}(w) w^{k}d\mu_j(w) 
        \\
        = -\int \varphi_{\bm{n}}(w)d\mu_j(w)+\mathcal{O}(z^{-\ell_j-1})
    \end{multline}
    (note that the term with $\int \varphi_{\bm{n}}(w)d\mu_j(w)$ is not needed in~\eqref{eq:R0} as it is always in $\mathcal{O}(z^{\ell_j})$, but it is needed in~\eqref{eq:RInfty} for the cases when $n_j=0$). 

    Now denote $\widetilde{\psi}_{\bm{n},j}(z)$ to be the right-hand side of~\eqref{eq:second kind polynomials}.
    It is elementary to see that $\frac{w+z}{w-z}\big(p(w)-p(z)\big)$ is in $\operatorname{span}\big\{z^k\big\}_{k = -|\bm{\ell}|}^{|\bm{\ell}|}$, provided that $p(z)$ is in $\operatorname{span}\big\{z^k\big\}_{k = -|\bm{\ell}|}^{|\bm{\ell}|}$. This shows that $\widetilde{\psi}_{\bm{n},j}(z)$ is in $\operatorname{span}\big\{z^k\big\}_{k = -|\bm{\ell}|}^{|\bm{\ell}|}$. Now, using ~\eqref{eq:secondKindFunction} we can rewrite the right-hand side of~\eqref{eq:second kind polynomials} as
    \begin{equation}\label{eq:pseVsReven}
        \widetilde{\psi}_{\bm{n},j}(z)
        =
        R_{\bm{n},j}(z) - \varphi_{\bm{n}}(z) F_j(z) + \int \varphi_{\bm{n}}(w)d\mu_j(w).
    \end{equation}
    
    Then~\eqref{eq:R0} and~\eqref{eq:RInfty} imply 
    \begin{alignat}{3}
    \label{eq:two point hermite pade eq tilde}
    & \varphi_{\bm{n}}(z)F_j(z) + \widetilde\psi_{\bm{n},j}(z) = \mathcal{O}(z^{\ell_j}),  \qquad&&z \rightarrow 0, \\\label{eq:two point hermite pade eq 2 tilde}
    & \varphi_{\bm{n}}(z)F_j(z) + \widetilde\psi_{\bm{n},j}(z) = \mathcal{O}(z^{-\ell_j-1}),  \qquad&&z \rightarrow \infty.
\end{alignat}

    In the beginning of the proof we showed that this together with~\eqref{eq:two point hermite pade eq tilde}--\eqref{eq:two point hermite pade eq 2 tilde} necessarily implies that $\widetilde{\psi}_{\bm{n},j}(z)$ must be $\psi_{\bm{n},j}(z)$ that is uniquely determined by~\eqref{eq:PadeToOrthoEq1}--\eqref{eq:PadeToOrthoEq3}. This proves~\eqref{eq:second kind polynomials}.
    
    For the final statement, given $\varphi_{\bm{n}}$ and the corresponding $\psi_{\bm{n},j}$ as in~\eqref{eq:second kind polynomials}, define $R_{\bm{n},j}$ as in~\eqref{eq:secondKindFunction}. As above, we show that it satisfies~\eqref{eq:R0} and ~\eqref{eq:RInfty}, which then becomes ~\eqref{eq:two point hermite pade eq 1}--\eqref{eq:two point hermite pade eq 2}.
\end{proof}

    

 As as result, $\varphi_{\bm{n}}$ naturally relates to $\phi_{\bm{n}}$ from the previous sections. 
 
\begin{cor}\label{cor:HPvsMLOPUC}
    {
    Assume 
    all $n_j$'s are even for  $j=1,\ldots,r$. Then the Hermite--Pad\'{e} problem \eqref{eq:span}, \eqref{eq:two point hermite pade eq 1}, \eqref{eq:two point hermite pade eq 2} 
    has a unique solution $\varphi_{\bm{n}}$ with $z^{\abs{\bm{n}}}$-coefficient equal to $1$ if and only if  $\bm{n}$ is $\phi$-normal (see Definition~\ref{def:normal B}) with respect to $\bm\mu$. It is given by $\varphi_{\bm{n}} = \phi_{\bm{n}}$.}
\end{cor}
\begin{rem}
    For the case of one measure $r=1$, we get that $\varphi_{\bm{n}} = \phi_{\bm{n}}$ is, of course, $z^{-n/2} \Phi_{n}(z)$ and Theorem~\ref{thm:hermite pade orthogonality} reduces to the usual two-point Padé approximation \cite{JNT}. Our version of the result particularly similar to that of Peherstorfer and Steinbauer~\cite{PehSte}. 
    It is natural to call $\psi_{\bm{n},j}$ the second kind (Laurent) multiple orthogonal polynomials, as for the $r=1$ case, see, e.g.,~\cite{OPUC1}.
\end{rem}

By taking reversal in \eqref{eq:two point hermite pade eq 1}-\eqref{eq:two point hermite pade eq 2} and using $F_j^\sharp(z) = -F_j(z)$ we obtain that the solution of the Hermite-Padé problem 
\begin{equation}\label{eq:span2}
 \varphi^\sharp_{\bm{n}},\psi^\sharp_{\bm{n},1},\ldots,\psi^\sharp_{\bm{n},r} \in\operatorname{span}\big\{z^k\big\}_{k = -|\bm{\ell}|}^{|\bm{\ell}|},   
\end{equation}
with
\begin{alignat}{3}
\label{eq:two point hermite pade eq 1 star}
    & \varphi_{\bm{n}}^\sharp(z)F_j(z) - \psi_{\bm{n},j}^\sharp(z) = \mathcal{O}(z^{\ell_j+1}),  \qquad&&z \rightarrow 0, \\ \label{eq:two point hermite pade eq 2 star}
    & \varphi_{\bm{n}}^\sharp(z)F_j(z) - \psi_{\bm{n},j}^\sharp(z) = \mathcal{O}(z^{-\ell_j}),  \qquad&&z \rightarrow \infty,
\end{alignat}
is solved by $\varphi^\sharp_{\bm{n}}(z) = \overline{\varphi_{\bm{n}}(1/\bar{z})}$, and (see  \eqref{eq:second kind polynomials}) $\psi_{\bm{n},j}^\sharp$ is given explicitly by
\begin{equation}
    \psi_{\bm{n},j}^\sharp(z) =  \int \frac{w+z}{w-z}\big(\varphi_{\bm{n}}^\sharp(z)-\varphi_{\bm{n}}^\sharp(w) \big)d\mu_j(w) +\int \varphi_{\bm{n}}^\sharp(w)d\mu_j(w). 
\end{equation}



\subsection{The case of all odd \texorpdfstring{$n_j$'s}{n_j's}, and odd 
\texorpdfstring{$|\bm{n}|$}{$|n|$}}\label{ss:HPphiOdd1}

\noindent

Now let us assume that all $n_j$'s are odd, $j=1,\ldots,r$, and $|\bm{n}|$ is also odd (i.e., $r$ is odd).
Let again $\bm{\ell} := \bm{n}/2$, so that $\ell_j = n_j/2$. Since each $n_j$ is odd, we get that $\ell_j \in \N+\tfrac12:= \{l+\tfrac12:l\in\N\}$ and $|\bm{\ell}|\in \N+\tfrac12$. 
With any choice of a branch of the square root function, we can still consider the Hermite--Pad\'{e} problem ~\eqref{eq:span},~\eqref{eq:two point hermite pade eq 1},~\eqref{eq:two point hermite pade eq 2}. The notation $\operatorname{span}\{z^k\}_{k = -|\bm{\ell}|}^{|\bm{\ell}|}$ in~\eqref{eq:span} stands for $\{z^{-|\bm{\ell}|},z^{-|\bm{\ell}|+1},\ldots,z^{|\bm{\ell}|}\}$, of course. Note that half-powers on the left-hand and right-hand sides of ~\eqref{eq:two point hermite pade eq 1},~\eqref{eq:two point hermite pade eq 2} can be canceled, if one wishes to make sense of the problem analytically rather than algebraically. 

The following analogue of Theorem~\ref{thm:hermite pade orthogonality} holds true.

\begin{thm}\label{thm:hermite pade orthogonalityOdd1}
     Assume that each $n_j$ is odd for $j=1,\ldots,r$, and $|\bm{n}|$ is odd. {Any} 
    $\varphi_{\bm{n}}$ that solves the Hermite--Pad\'{e} problem~\eqref{eq:span},~\eqref{eq:two point hermite pade eq 1},~\eqref{eq:two point hermite pade eq 2}, satisfies the orthogonality relations~\eqref{eq:hermite pade orthogonality}.
    Moreover, 
    $\psi_{\bm{n},j}$ can be expressed in terms of $\varphi_{\bm{n}}$ by
    \begin{equation}\label{eq:second kind polynomialsOdd1}
        \psi_{\bm{n},j}(z) = z^{-1/2} \int  \frac{w+z}{w-z} \big(w^{1/2}\varphi_{\bm{n}}(w)-z^{1/2}\varphi_{\bm{n}}(z)\big)d\mu_j(w). 
    \end{equation}

    Conversely, any $\varphi_{\bm{n}} \in \operatorname{span}\set{z^k}_{k = -|\bm{\ell}|}^{|\bm{\ell}|}$, that satisfies the orthogonality relations~\eqref{eq:hermite pade orthogonality}, solves the Hermite--Pad\'{e} problem ~\eqref{eq:two point hermite pade eq 1}--\eqref{eq:two point hermite pade eq 2} together with~\eqref{eq:second kind polynomialsOdd1}.
\end{thm}
\begin{rem}
    Easy to see that Corollary~\ref{cor:HPvsMLOPUC} holds true for the current setting as well.
\end{rem}
\begin{proof}
    The proof of~\eqref{eq:hermite pade orthogonality} works as in Theorem~\ref{thm:hermite pade orthogonality} with no modifications. Let us prove~\eqref{eq:second kind polynomialsOdd1} now.
    
    We replace~\eqref{eq:secondKindFunction} with
    \begin{equation}\label{eq:secondKindFunctionOdd}
        \hat{R}_{\bm{n},j}(z) = \int \frac{w+z}{w-z} w^{1/2}\varphi_{\bm{n}}(w) d\mu_j(w).
    \end{equation}
    Using \eqref{eq:hermite pade orthogonality} and~\eqref{eq:kernel}, we get
        \begin{equation}\label{eqRodd1}
        \hat{R}_{\bm{n},j}(z)
        =
        \begin{cases}
            2\sum_{k=1}^\infty z^k \int \varphi_{\bm{n}}(w) w^{-(k-1/2)} d\mu_j(w) = \mathcal{O}(z^{\ell_j+1/2}), &  z\to 0, \\
            - 2\sum_{k=1}^\infty z^{-k} \int \varphi_{\bm{n}}(w) w^{k+1/2} d\mu_j(w) = \mathcal{O}(z^{-\ell_j-1/2}), &  z\to \infty,
        \end{cases}
    \end{equation}
    where we used that $n_j\ge 1$, so that $\ell_j\ge 1/2$.

    Now denote $\widetilde{\psi}_{\bm{n},j}(z)$ to be the right-hand side of~\eqref{eq:second kind polynomialsOdd1}.   Since $z^{1/2}\varphi_{\bm{n}}(z)$ is a Laurent polynomial in $\operatorname{span}\big\{z^k\big\}_{k = -|\bm{\ell}|+1/2}^{|\bm{\ell}|+1/2}$, we obtain that $\widetilde{\psi}_{\bm{n},j}(z)$ is in $\operatorname{span}\big\{z^k\big\}_{k = -|\bm{\ell}|}^{|\bm{\ell}|}$. We rewrite the right-hand side of~\eqref{eq:second kind polynomialsOdd1} as
    \begin{equation}\label{eq:tildepsiOdd1}
        \widetilde{\psi}_{\bm{n},j}(z) = z^{-1/2} \hat{R}_{\bm{n},j}(z)  -  \varphi_{\bm{n}}(z) F_j(z). 
    \end{equation}
    Combining ~\eqref{eq:tildepsiOdd1} with~\eqref{eqRodd1} we obtain that $\varphi_{\bm{n}}$ and $\widetilde{\psi}_{\bm{n},j}$ satisfy~\eqref{eq:two point hermite pade eq tilde}--\eqref{eq:two point hermite pade eq 2 tilde}. Now the rest of the proof of Theorem~\ref{thm:hermite pade orthogonality} goes through without changes. 
\end{proof}

\subsection{The case of all odd \texorpdfstring{$n_j$'s}{n_j's}, and even 
\texorpdfstring{$|\bm{n}|$}{$|n|$}}\label{ss:HPphiOdd2}

\noindent

In this section, we assume that all $n_j$'s are odd, $j=1,\ldots,r$, but $|\bm{n}|$ is even (i.e., $r$ is even). Equivalenty, each $|\bm{n}|-n_j$ is odd (note that in Sections~\ref{ss:HPphiEven} and~\ref{ss:HPphiOdd1} each $|\bm{n}|-n_j$ was even). 
As usual $\bm{\ell} := \bm{n}/2$, so that $\ell_j = n_j/2$. Since each $n_j$ is odd, we get that $\ell_j \in \N+\tfrac12$, but now $|\bm{\ell}| = |\bm{n}|/2\in \N$.

With a chosen branch of the square root function $z^{1/2}$, let
\begin{equation}\label{eq:G}
    G_j(z) =  \int  \frac{w^{1/2}z^{1/2}}{w-z} d\mu_j(w)
\end{equation}
for each $j=1,\ldots,r$. These functions can be expanded into
\begin{alignat}{3}
    \label{eq:G0}
     & G_j(z) = \sum_{k=1/2}^\infty c_{k,j} z^{k},\qquad&&  \abs{z} < 1, \\
        \label{eq:Ginf}
    & G_j(z) = -\sum_{k=1/2}^{\infty} c_{-k,j} z^{-k}, \qquad&& \abs{z} > 1,
\end{alignat}
where $c_{k,j}$ are given by~\eqref{eq:moments of measures} but now
with $k$ in $\Z+\tfrac12:=\{l+\tfrac12:l\in\Z\}$.

We consider the following Hermite--Pad\'{e} type problem: find functions of the form
     \begin{align}
     \label{eq:phisOdd}
     \varphi_{\bm{n}} & \in \operatorname{span}\big\{z^k\big\}_{k = -|\bm{\ell}|}^{|\bm{\ell}|}, 
     \\
     \label{eq:psisOdd}
     \psi_{\bm{n},j} & \in \operatorname{span}\big\{ z^k\big\}_{k = -|\bm{\ell}|+1/2}^{|\bm{\ell}|-1/2}, 
     \end{align}
 that solve 
    \begin{alignat}{3}
    \label{eq:two point hermite pade eq 3+}
    & \varphi_{\bm{n}}(z)G_j(z) + \psi_{\bm{n},j}(z) = \mathcal{O}(z^{\ell_j}), \qquad&&z \rightarrow 0, 
    \\
    \label{eq:two point hermite pade eq 4+}
    & \varphi_{\bm{n}}(z)G_j(z) + \psi_{\bm{n},j}(z) = \mathcal{O}(z^{-\ell_j-1}), \qquad&&z \rightarrow \infty. 
\end{alignat}

\begin{thm}\label{thm:hermite pade orthogonalityOdd2}
     Assume that each $n_j$ is odd for $j=1,\ldots,r$, and $|\bm{n}|$ is even. {Any} 
    Laurent polynomial 
    $\varphi_{\bm{n}}$ that solves the Hermite--Pad\'{e} problem~\eqref{eq:phisOdd},~\eqref{eq:psisOdd}, \eqref{eq:two point hermite pade eq 3+},~\eqref{eq:two point hermite pade eq 4+}, satisfies the orthogonality relations~\eqref{eq:hermite pade orthogonality}.
    Moreover, $\psi_{\bm{n},j}$ can be expressed in terms of $\varphi_{\bm{n}}$ by
    \begin{equation}\label{eq:second kind polynomialsOdd2}
        \psi_{\bm{n},j}(z) = \int  \frac{w^{1/2}z^{1/2}}{w-z} \big(\varphi_{\bm{n}}(w)-\varphi_{\bm{n}}(z)\big)d\mu_j(w). 
    \end{equation}

    Conversely, any $\varphi_{\bm{n}} \in \operatorname{span}\set{z^k}_{k = -|\bm{\ell}|}^{|\bm{\ell}|}$, that satisfies the orthogonality relations~\eqref{eq:hermite pade orthogonality}, solves the Hermite--Pad\'{e} problem ~\eqref{eq:phisOdd},~\eqref{eq:psisOdd},~\eqref{eq:two point hermite pade eq 3+},~\eqref{eq:two point hermite pade eq 4+} together with~\eqref{eq:second kind polynomialsOdd2}.
\end{thm}
\begin{rem}
    Again, it is easy to see that Corollary~\ref{cor:HPvsMLOPUC} holds true for the current setting as well.
\end{rem}
\begin{proof}
    The proof follows similar lines to that of Theorem~\ref{thm:hermite pade orthogonality} with some minor modifications related to the half-powers. Write $\varphi_{\bm{n}}$ and $\psi_{\bm{n},j}$ in the form 
    \begin{align}\label{eq:coefficients of chi}
         \varphi_{\bm{n}}(z) &= \kappa_{|\bm{\ell}|}z^{|\bm{\ell}|} + \kappa_{|\bm{\ell}|-1}z^{|\bm{\ell}|-1}+\dots + \kappa_{-|\bm{\ell}|}z^{-|\bm{\ell}|}, \\ \label{eq:coefficients of y}
        \psi_{\bm{n},j}(z) & = \lambda_{|\bm{\ell}|-1/2,j}z^{|\bm{\ell}|-1/2} + \lambda_{|\bm{\ell}|-3/2,j}z^{|\bm{\ell}|-3/2}+ \dots + \lambda_{-|\bm{\ell}|+1/2,j}z^{-|\bm{\ell}|+1/2}.
    \end{align}
    We then have 
    \begin{align*}
        & \varphi_{\bm{n}}(z)G_j(z) = a_{-|\bm{\ell}|+1/2,j}z^{-|\bm{\ell}|+1/2} + \dots + a_{\ell_j-1,j}z^{\ell_j-1} + \mathcal{O}(z^{\ell_j}),  &z& \rightarrow 0, \\
        & \varphi_{\bm{n}}(z)G_j(z) = b_{|\bm{\ell}|-1/2,j}z^{|\bm{\ell}|-1/2} + \dots + b_{-\ell_j,j}z^{-\ell_j} + \mathcal{O}(z^{-\ell_j-1}),  &z& \rightarrow \infty,
    \end{align*}
    where the coefficients are given by
    \begin{align}\label{eq:coefficients of series 1odd}
         & a_{k,j} = \kappa_{k-1/2}c_{1/2,j} + \dots + \kappa_{-\abs{\bm{\ell}}} c_{\abs{\bm{\ell}}+k,j},  &k& = -|\bm{\ell}|+\tfrac12,\dots,\ell_j-1, 
         \\ 
         \label{eq:coefficients of series 2odd}
         & b_{k,j} = - \kappa_{k+1/2} c_{-1/2,j} - \dots - \kappa_{\abs{\bm{\ell}}} c_{-\abs{\bm{\ell}}+k,j}, &k& = -\ell_j,\dots,|\bm{\ell}|-\tfrac12.
    \end{align}
    To get \eqref{eq:two point hermite pade eq 3+}-\eqref{eq:two point hermite pade eq 4+} we necessarily need
    \begin{alignat}{3}
        \label{eq:PadeToOrthoEq1odd}
        & -\lambda_{k,j} = a_{k,j}, \qquad&& k = -|\bm{\ell}|+\tfrac12,\dots,-\ell_j-1, \\ 
        \label{eq:PadeToOrthoEq2odd}
        & -\lambda_{k,j} = a_{k,j}=b_{k,j},  \qquad&& k = -\ell_j,\dots,\ell_j-1, \\ 
        \label{eq:PadeToOrthoEq3odd}
        & -\lambda_{k,j} = b_{k,j},  \qquad&&k  = \ell_j,\dots,|\bm{\ell}|-\tfrac12.
    \end{alignat}
    Then~\eqref{eq:PadeToOrthoEq2odd}, combined with~\eqref{eq:coefficients of series 1odd}--\eqref{eq:coefficients of series 2odd} and~\eqref{eq:moments of measures}, turns into
    \begin{equation}
         \int \big(\kappa_{|\bm{\ell}|}z^{|\bm{\ell}|} + \dots + \kappa_{-|\bm{\ell}|}z^{-|\bm{\ell}|}\big)z^{-k}d\mu_j(z) = 0, \qquad k = -\ell_j,\dots,\ell_j-1.
    \end{equation}

    This proves \eqref{eq:hermite pade orthogonality}.  Now denote
    \begin{equation}\label{eq:secondKindFunctionOdd}
        \widetilde{R}_{\bm{n},j}(z) = \int \frac{w^{1/2}z^{1/2}}{w-z} \varphi_{\bm{n}}(w) d\mu_j(w).
    \end{equation}
    Note that
    \begin{equation}
        \frac{w^{1/2}z^{1/2}}{w-z}
        =
        \begin{cases}
            \sum_{k=1/2}^\infty \frac{z^k}{w^k},\quad  & \mbox{if } |z|<|w|, \\
            - \sum_{k=1/2}^\infty \frac{w^k}{z^k}, \quad & \mbox{if } |z|>|w|,
        \end{cases}
    \end{equation}
    where the summation is over the set $\N+\tfrac12$. Combining this with \eqref{eq:hermite pade orthogonality}, we get
        \begin{equation}\label{eqRodd}
        \widetilde{R}_{\bm{n},j}(z)
        =
        \begin{cases}
            \sum_{k=1/2}^\infty z^k \int \varphi_{\bm{n}}(w) w^{-k} d\mu_j(w) = \mathcal{O}(z^{\ell_j}), & z\to 0, \\
            - \sum_{k=1/2}^\infty z^{-k} \int \varphi_{\bm{n}}(w) w^{k} d\mu_j(w) = \mathcal{O}(z^{-\ell_j-1}), & z\to \infty.
        \end{cases}
    \end{equation}

    Now denote $\widetilde{\psi}_{\bm{n},j}(z)$ to be the right-hand side of~\eqref{eq:second kind polynomialsOdd2}, that is,
    \begin{equation}\label{eq:psiVsR}
        \widetilde{\psi}_{\bm{n},j}(z) = \widetilde{R}_{\bm{n},j}(z)  -  \varphi_{\bm{n}}(z) G_j(z). 
    \end{equation}
    We can rewrite~\eqref{eqRodd} then as 
\begin{alignat}{3}
\label{eq:two point hermite pade eq tildeOdd}
    & \varphi_{\bm{n}}(z)G_j(z) + \widetilde\psi_{\bm{n},j}(z) = \mathcal{O}(z^{\ell_j}),  \qquad&&z \rightarrow 0, \\\label{eq:two point hermite pade eq 2 tildeOdd}
    & \varphi_{\bm{n}}(z)G_j(z) + \widetilde\psi_{\bm{n},j}(z) = \mathcal{O}(z^{-\ell_j-1}),  \qquad&&z \rightarrow \infty.
\end{alignat}

    Finally note that if $\varphi_{\bm{n}}(z)$ is a Laurent polynomial in $\operatorname{span}\big\{z^k\big\}_{k = -|\bm{\ell}|}^{|\bm{\ell}|}$, then $\widetilde{\psi}_{\bm{n},j}(z)$, the right-hand side of~\eqref{eq:second kind polynomialsOdd2}, is in $\operatorname{span}\big\{ z^k\big\}_{k = -|\bm{\ell}|+1/2}^{|\bm{\ell}|-1/2}$. Now the rest of the proof of Theorem~\ref{thm:hermite pade orthogonality} goes through without changes. 
\end{proof}


\subsection{The case of any \texorpdfstring{$\bm{n}\in\N^r$}{n} }
\label{ss:HPphiFull}

\noindent

Let $\bm{n}\in\N^r$ be arbitrary now, and set $\bm\ell=\bm{n}/2$ as usual. 

Let us consider the following hybrid approximation problem. We let $\varphi_{\bm{n}}\in \operatorname{span}\set{z^k}_{k = -|\bm{\ell}|}^{|\bm{\ell}|}$. For those $j=1,\ldots,r$ that have $(|\bm{n}|-n_j)\operatorname{mod} 2 = 0$, we approximate the Carath\'{e}odory function $F_j$ of $\mu_j$ as in~\eqref{eq:two point hermite pade eq 1}--\eqref{eq:two point hermite pade eq 2} with $\psi_{\bm{n},j}\in \operatorname{span}\set{z^k}_{k = -|\bm{\ell}|}^{|\bm{\ell}|}$. For those $j=1,\ldots,r$ that has $(|\bm{n}|-n_j)\operatorname{mod} 2 = 1$, we approximate the $G_j$ function ~\eqref{eq:G} of $\mu_j$ as in~\eqref{eq:two point hermite pade eq 3+}--\eqref{eq:two point hermite pade eq 4+} with $\psi_{\bm{n},j}$ as in~\eqref{eq:psisOdd}. Note that the ``denominator'' function $\varphi_{\bm{n}}$ is common for all $r$ approximations.

\begin{cor}\label{cor:HPvsMLOPUCFull}
    {
    Let $\bm{n}\in\N^r$ be arbitrary. Then this approximation problem 
    has a unique monic solution $\varphi_{\bm{n}}$ if and only if  $\bm{n}$ is $\phi$-normal (see Definition~\ref{def:normal B}) with respect to $\bm\mu$. It is given by $\varphi_{\bm{n}} = \phi_{\bm{n}}$.}
\end{cor}
\begin{proof}
Note that the proofs of Theorems~\ref{thm:hermite pade orthogonality}, \ref{thm:hermite pade orthogonalityOdd1}, and~\ref{thm:hermite pade orthogonalityOdd2} only depended on the parity of $|\bm{n}|$ and of $n_j$ for a chosen $j$ (i.e., it does not depend on the parity of $n_l$ for $l\ne j$).  
\end{proof}
\begin{rem}
By combining this with Theorems~\ref{thm:Angelesco} and~\ref{thm:AT}, we see that this approximation problem has a unique solution, up to a multiplicative normalization, if $\bm{\mu}$ is any Angelesco or any AT system $\mu$  for any multi-index $\bm{n}\in\N^r$. 
\end{rem}

\section{A two-point Hermite--Padé Problem associated with \texorpdfstring{$\xi_{\bm{n}}$}{xi_n}}\label{ss:HPtypeI}

Let us now show that the type I functions $\xi_{\bm{n}}$ satisfy an alternative two-point Hermite-Padé problem. For simplicity let us restrict ourselves to the case when $\bm{n}\in\N^r$ has all $n_j$'s even. The other cases can be handled in a similar way as for type II functions in the previous section. As before, let $\bm{\ell}=(\ell_1,\ldots,\ell_r)\in\N^r$ with $\bm{n} = 2\bm{\ell}$, so that, in particular, $|\bm{\ell}|=|\bm{n}|/2 \in \N$ and $\ell_j= n_j/2\in\N$. Denote also $L:=\max\{\ell_j\}$.

We pose a problem of finding Laurent polynomials
\begin{equation}    
\label{eq:HPtypeI3half}
 \xi_{\bm{n},j}(z) \in \operatorname{span}\big\{z^{k}\big\}_{k=-\ell_j}^{\ell_j-1},
    \qquad \upsilon_{\bm{n}}(z) \in \operatorname{span}\big\{z^{k}\big\}_{k=-L}^{L-1},
\end{equation}
that satisfy
\begin{alignat}{3}
    \label{eq:HPtypeI1half}
    & \sum_{j = 1}^r \xi_{\bm{n},j}(z)F_j(z) +\upsilon_{\bm{n}}(z) = \mathcal{O}(z^{\abs{\bm{\ell}}}), \qquad &&z \rightarrow 0, \\
    \label{eq:HPtypeI2half}
    &\sum_{j = 1}^r \xi_{\bm{n},j}(z)F_j(z) +\upsilon_{\bm{n}}(z) = \mathcal{O}(z^{-\abs{\bm{\ell}}}),  &&z \rightarrow \infty,
\end{alignat}
where $F_j$ are the Carath\'{e}odory functions~\eqref{eq:caratheodory function}. 

\begin{thm}
    If the Laurent polynomials 
    $\bm{\xi}_{\bm{n}} = (\xi_{\bm{n},1},\dots,\xi_{\bm{n},r})$ and $\upsilon_{\bm{n}}$ solve~\eqref{eq:HPtypeI3half}, \eqref{eq:HPtypeI1half}, \eqref{eq:HPtypeI2half}, then $\bm{\xi}_{\bm{n}}$ satisfies the orthogonality relations
    \begin{equation}\label{eq:orthogonality type I laurenthalf}
        \sum_{j = 1}^r \int \xi_{\bm{n},j}(w)w^{-k} d\mu_j(w) = 0, \qquad k = -\abs{\bm{\ell}}+1,\dots,\abs{\bm{\ell}}-1,
    \end{equation}
    and $\upsilon_{\bm{n}}$ is given by
    \begin{equation}\label{eq:upsilonhalf}
    \upsilon_{\bm{n}}(z) = \sum_{j=1}^r \int \frac{w+z}{w-z}\big(\xi_{\bm{n},j}(w)-\xi_{\bm{n},j}(z)\big)d\mu_j(w),
\end{equation}

    Conversely, any vector of Laurent polynomials $\bm{\xi}_{\bm{n}} = (\xi_{\bm{n},1},\dots,\xi_{\bm{n},r})$ with $\xi_{\bm{n},j} \in \operatorname{span}\set{z^k}_{k = -\ell_j}^{\ell_j-1}$, that satisfies 
    ~\eqref{eq:orthogonality type I laurenthalf}, solves the Hermite--Pad\'{e} problem ~\eqref{eq:HPtypeI1half}, \eqref{eq:HPtypeI2half} together with the Laurent polynomial given by~\eqref{eq:upsilonhalf}.
\end{thm}
\begin{proof}
    We let 
    \begin{align}\label{eq:lamdbahalf}
         \xi_{\bm{n},j} (z)&= \kappa_{\ell_j-1,j}z^{\ell_j-1} + \kappa_{\ell_{j-2},j}z^{\ell_j-2}+\dots + \kappa_{-\ell_j,j}z^{-\ell_j}, \\ 
         \label{eq:xihalf}
        \upsilon_{\bm{n}}(z) & = \lambda_{L-1}z^{L-1} + \lambda_{L-2}z^{L-2}+ \dots + \lambda_{-L}z^{-L}.
    \end{align}
    Let us put $\kappa_{k,j} = 0$ if $k< -\ell_j$ or $k> \ell_j-1$, and extend \eqref{eq:lamdbahalf} to bi-infinite power series. We then have 
    \begin{align*}
        & \xi_{\bm{n},j} (z)F_j(z) = a_{-\ell_j,j}z^{-\ell_j} + \dots + a_{|\bm{\ell}|-1,j}z^{|\bm{\ell}|-1} + \mathcal{O}(z^{|\bm{\ell}|}),  &z& \rightarrow 0, \\
        & \xi_{\bm{n},j} (z) F_j(z) = b_{\ell_j-1,j}z^{\ell_j-1} + \dots + b_{-|\bm{\ell}|+1,j}z^{-|\bm{\ell}|+1} + \mathcal{O}(z^{-|\bm{\ell}|}),  &z& \rightarrow \infty,
    \end{align*}
    where 
    \begin{align}\label{eq:coefficients of series 1 Ihalf}
         & a_{k,j} = \kappa_{k,j} c_{0,j}+ 2\kappa_{k-1,j}c_{1,j} + \ldots + 2\kappa_{-\ell_j,j} c_{\ell_j+k,j},  
         \\ 
         \label{eq:coefficients of series 2 Ihalf}
         & b_{k,j} = -\kappa_{k,j} c_{0,j} - 2\kappa_{k+1,j} c_{-1,j} - \ldots - 2 \kappa_{\ell_j-1,j} c_{-\ell_j+k+1,j}
         .
    \end{align}

    If \eqref{eq:HPtypeI1half}--\eqref{eq:HPtypeI2half} hold true, then    
    \begin{alignat}{3}
        \label{eq:PadeToOrthoEq Ihalf AGAIN1}
        \sum_{j=1}^r a_{k,j} & + \lambda_k = 0, \qquad && k\le |\bm{\ell}|-1,
        \\
        \label{eq:PadeToOrthoEq Ihalf AGAIN2}
        \sum_{j=1}^r b_{k,j} & + \lambda_k = 0, \qquad && k\ge |\bm{\ell}|-1.
    \end{alignat}
    This implies
     \begin{equation}
         \label{eq:PadeToOrthoEq Ihalf}
           \sum_{j=1}^r a_{k,j} = \sum_{j=1}^r b_{k,j},  \quad k = -|\bm{\ell}|+1,\dots,|\bm{\ell}|-1,
     \end{equation}
    which, combined with~\eqref{eq:coefficients of series 1 Ihalf}--\eqref{eq:coefficients of series 2 Ihalf} and~\eqref{eq:moments of measures}, implies~\eqref{eq:orthogonality type I laurenthalf}.
    

    Observe also that if $\bm{\xi}_{\bm{n}}$ is determined, then~\eqref{eq:HPtypeI1half}--\eqref{eq:HPtypeI2half} uniquely determines all the coefficients $\lambda_j$. Indeed, \eqref{eq:PadeToOrthoEq Ihalf AGAIN1}--\eqref{eq:PadeToOrthoEq Ihalf AGAIN2} determine all $\lambda_j$ for all $j$. Note that since $a_{k,j} = 0$ if $k < -L $ and $b_{k,j} = 0$ if $k > L-1$, we get that $\upsilon_{\bm{n}}$ is of the form \eqref{eq:xihalf}. 

    Denote
    \begin{equation}\label{eq:RtypeIhalf}
        R_{\bm{n}}(z) = \sum_{j=1}^r \int \frac{w+z}{w-z} \xi_{\bm{n},j}(w) d\mu_j(w).
    \end{equation}
    Then~\eqref{eq:kernel} and~\eqref{eq:orthogonality type I laurenthalf} together imply 
    \begin{equation}
    \label{eq:Rhalf}
    R_{\bm{n}}(z) =
    \begin{cases}
        \mathcal{O}(z^{|\bm{\ell}|}), & z \to 0,
        \\
        \mathcal{O}(z^{-|\bm{\ell}|}), & z \to \infty.
    \end{cases}
    \end{equation}

    Let $\widetilde{\upsilon}_{\bm{n}}(z)$ be the right-hand side of~\eqref{eq:upsilonhalf}. It can be rewritten as  
    \begin{equation*}
        \widetilde{\upsilon}_{\bm{n}}(z)
        =
        R_{\bm{n}}(z) - \sum_{j=1}^r \xi_{\bm{n},j}(z) F_j(z).
    \end{equation*}
        
    Combining this with~\eqref{eq:Rhalf}, we see that
\eqref{eq:HPtypeI1half}--\eqref{eq:HPtypeI2half} hold with the same $\bm{\xi}_{\bm{n}}$ but with $\upsilon_{\bm{n}}$ replaced by $\widetilde\upsilon_{\bm{n}}$. It is clear that $\widetilde{\upsilon}_{\bm{n}}(z)$ is a polynomial in $\operatorname{span}\big\{z^k\big\}_{k = -L}^{L-1}$.  
    But the proof showed that \eqref{eq:HPtypeI3half}, \eqref{eq:HPtypeI1half}, \eqref{eq:HPtypeI2half} uniquely determine all the coefficients of $\widetilde\upsilon_{\bm{n}}$. This shows that $\widetilde\upsilon_{\bm{n}} = \upsilon_{\bm{n}}$ and proves \eqref{eq:upsilonhalf}.

    Conversely, given $\bm{\xi}_{\bm{n}}$ with $\xi_{\bm{n},j} \in \operatorname{span}\set{z^k}_{k = -\ell_j}^{\ell_j-1}$, that satisfies ~\eqref{eq:orthogonality type I laurenthalf}, we define $\upsilon_{\bm{n}}$ as in~\eqref{eq:upsilonhalf} and $R_{\bm{n}}$ as in~\eqref{eq:RtypeIhalf}. As above, we show that it satisfies~\eqref{eq:Rhalf}, which then becomes ~\eqref{eq:HPtypeI1half}, \eqref{eq:HPtypeI2half}.
\end{proof}

\begin{cor}\label{cor:HPvsMLOPUCtypeI}
    Assume 
    all $n_j$'s are even for  $j=1,\ldots,r$. Then ~\eqref{eq:HPtypeI3half}, \eqref{eq:HPtypeI1half}, \eqref{eq:HPtypeI2half}   
    has a unique solution $\bm\xi_{\bm{n}}$ satisfying the normalization~\eqref{eq:type I def SWAPPED} if and only if  $\bm{n}$ is $\phi$-normal (see Definition~\ref{def:normal B}) with respect to $\bm\mu$. 
    The solutions are then multiples of the vector defined by \eqref{eq:type I coefficients}-\eqref{eq:type I def SWAPPED}.
\end{cor}

By taking reversal in \eqref{eq:HPtypeI1half}-\eqref{eq:HPtypeI2half} and using $F_j^\sharp(z) = -F_j(z)$ we obtain that the solution of the Hermite-Padé problem
\begin{align}
    & \xi^\sharp_{\bm{n},j}(z) \in \operatorname{span}\big\{z^{k}\big\}_{k=-\ell_j+1}^{\ell_j},
    & \upsilon^\sharp_{\bm{n}}(z) \in \operatorname{span}\big\{z^{k}\big\}_{k=-L+1}^{L},
\end{align}
with
\begin{alignat}{3}
    & \sum_{j = 1}^r \xi^\sharp_{\bm{n},j}(z)F_j(z) - \upsilon^\sharp_{\bm{n}}(z) = \mathcal{O}(z^{\abs{\bm{\ell}}}), \qquad &&z \rightarrow 0, \\
    &\sum_{j = 1}^r \xi^\sharp_{\bm{n},j}(z)F_j(z) -\upsilon^\sharp_{\bm{n}}(z) = \mathcal{O}(z^{-\abs{\bm{\ell}}}),  &&z \rightarrow \infty,
\end{alignat}
is solved by $\bm\xi^\sharp_{\bm{n}}(z) = \overline{\bm{\xi}_{\bm{n}}(1/\bar{z})}$, and $\upsilon_{\bm{n}}^\sharp$ is given explicitly by
\begin{equation}
        \upsilon^\sharp_{\bm{n}}(z) = \sum_{j = 1}^r \int \frac{w+z}{w-z}\big(\xi^\sharp_{\bm{n},j}(z) - \xi^\sharp_{\bm{n},j}(w) \big) d\mu_j(w). 
    \end{equation}

\bibsection



\end{document}